\documentclass[a4paper]{article} 
\usepackage[latin1]{inputenc}
\usepackage[english]{babel} 
\usepackage{t1enc}
\usepackage{pdfsync}
\usepackage{float}
\usepackage{theorem,xspace}
\usepackage{graphicx} 
\usepackage{array} 
\usepackage{mathrsfs} 
\usepackage{multirow}
\usepackage{rotating} 
\usepackage{amsmath,amssymb}
\usepackage{xypic} 
\usepackage[heavycircles]{stmaryrd}
\usepackage[active]{srcltx}
\usepackage[all]{xy}
\usepackage[usenames,dvipsnames]{color}

\theorembodyfont{\rmfamily}

\advance\textheight by 2cm
\advance\topmargin -1cm

\advance\textwidth by 2cm
\advance\oddsidemargin by -1cm
\advance\evensidemargin by -1cm

\newcounter{noqed}
\newcommand{\qed}{ \ifmmode\mbox{ }\fi\rule[-.05em]{.3em}{.7em}\setcounter{noqed}{0}}
\newenvironment{proof}[1][{}]{\noindent{\bf Proof#1. }\setcounter{noqed}{1}}{\ifnum\value{noqed}=1\qed\fi\par\medskip}
\newcommand{\ep}[1]{\left\llbracket #1\right\rrbracket}

\newfloat{Algorithm}{t}{alg}

\newcounter{prgline}
\newcommand{\pl}{\theprgline\addtocounter{prgline}{1}}

\providecommand{\comp}{\circ}

\newcommand{\?}{\mskip1.5mu}
\newcommand{\singint}[1]{\left[#1\right]}
\newcommand{\sing}[1]{\left\{\?#1\?\right\}}

\newcommand{\crit}[1]{\mathscr C_{#1}}

\newcommand{\N}{\mathbf N}
\newcommand{\A}{\mathscr A}
\newcommand{\E}{\mathscr E}
\newcommand{\I}{\mathscr I}
\newcommand{\Inf}{\mathscr I^\infty}
\renewcommand{\L}{\mathscr L}

\newcommand{\cgen}[1]{\mathord{\text{\textasciitilde}}{#1}}

\newcommand{\downset}{\mathord\downarrow}
\newcommand{\upset}{\mathord\uparrow}

\newcommand{\himin}{t_{\mathrm{min}}}
\newcommand{\himax}{t_{\mathrm{max}}}

\renewcommand{\emptyset}{\varnothing}
\renewcommand{\epsilon}{\varepsilon}
\renewcommand{\phi}{\varphi}
\newcommand{\op}{{\operatorname{op}}}

\newcommand{\OR}{\vee}
\newcommand{\AND}{\wedge}

\newcommand{\ORirr}{$\OR$-irreducible\xspace}
\newcommand{\ANDirr}{$\AND$-irreducible\xspace}
\newcommand{\ORprime}{$\OR$-prime\xspace}
\newcommand{\ANDprime}{$\AND$-prime\xspace}
\newcommand{\ORrepr}{$\OR$-representation\xspace}
\newcommand{\ANDrepr}{$\AND$-representation\xspace}

\newcommand{\notstrcont}{\mathbin{\not\mathrel\rhd}}
\newcommand{\CCB}{Clarke--Cormack--Burkowski\xspace}

\newcommand{\sqleq}{\sqsubseteq}
\newcommand{\sqgeq}{\sqsupseteq}
\newcommand{\sqless}{\sqsubset}

\def\..{\,\mathpunct{\ldotp\ldotp}} 
\newcommand{\rinf}[1]{[#1\..\rightarrow)}
\newcommand{\linf}[1]{(\leftarrow\..#1]}

\newtheorem{theorem}{Theorem}
\newtheorem{definition}{Definition}
\newtheorem{proposition}{Proposition}
\newtheorem{lemma}{Lemma}
\newtheorem{corollary}{Corollary}

\newcommand{\BEGIN}{\text{\textbf{begin}}\xspace}
\newcommand{\FOR}{\text{\textbf{for}}\xspace}

\newcommand{\TO}{\text{\textbf{to}}\xspace}
\newcommand{\END}{\text{\textbf{end}}\xspace}
\newcommand{\DO}{\text{\textbf{do}}\xspace}
\newcommand{\OD}{\text{\textbf{od}}\xspace}

\newcommand{\PROCEDURE}{\text{\textbf{procedure}}\xspace}

\title{On the lattice of antichains of finite intervals}

\author{Paolo Boldi\thanks{Dipartimento di Informatica,
    Universit\`a degli Studi di Milano, via Comelico 39, I-20135 Milano, Italy.
  \texttt{paolo.boldi@unimi.it}.}
\and Sebastiano Vigna\thanks{Contact author. Dipartimento di Informatica,
    Universit\`a degli Studi di Milano, via Comelico 39, I-20135 Milano, Italy.
  \texttt{sebastiano.vigna@unimi.it}.}}

\date{}

\begin{document}

\bibliographystyle{plain}

\maketitle

\begin{abstract}
Motivated by applications to information retrieval, we study the lattice of
antichains of finite intervals of a locally finite, totally ordered set. 
Intervals are ordered by reverse inclusion; the order between antichains is
induced by the lower set they generate. We discuss in general properties of such
\emph{antichain completions}; in particular, their connection with
Alexandrov completions. We prove the existence of a unique, irredundant
\ANDrepr by \ANDirr elements, which makes it possible to write 
the relative pseudo-complement in closed form. We also discuss in detail
properties of additional interesting operators used in information retrieval.
Finally, we give a formula for the rank of an element and for the height of the lattice.
\end{abstract}

\section{Introduction}
\label{sec:intro}

Modern information-retrieval systems, such as web search engines, rely on
different models to compute the answer to a query. The simplest one is 
the \emph{Boolean model}, in which operators are just conjunction,
disjunction and negation, and the answer is just ``relevant'' or ``not relevant''.
A richer model is given by \emph{minimal-interval semantics}~\cite{CCBASTSFI}, which
uses \emph{antichains} (w.r.t.~inclusion) of \emph{intervals} (i.e., textual
passages in the document) to represent the answer to a query; this is a
very natural framework in which operators such as ordered conjunction, proximity
restriction, etc., can be defined and combined freely. Each interval in the text of a document is a \emph{witness} of the satisfiability of
the query, that is, it represents a region of the document that satisfies the
query.
Words in the document are numbered (starting from $0$), so regions of text are
identified with integer intervals, that is, sets of consecutive natural numbers.
For example, a query formed by the conjunction of two terms is satisfied by the
minimal intervals representing the regions of the document that contain both
terms.
These intervals can be used not only to estimate the relevance of the document to the query~\cite{ClCSSRR}, but also to provide the user with
\emph{snippets}---fragments of texts witnessing where (and why) the document
satisfies the query.

Clarke, Cormack and Burkowski defined minimal-interval semantics in their
seminal work~\cite{CCBASTSFI}, but they missed the connection with lattice
theory, proving every single property (e.g., distributivity) from scratch: in
this paper, we firstly aim at providing a principled introduction of
their framework in terms of \emph{antichain completions} and their relation with
Alexandrov completions. On the other hand, we extend the study of these
structures by allowing infinite antichains, and characterizing in an elementary
manner operators, such as the relative pseudo-complement, some of which 
have an immediate interpretation in information retrieval.

We conclude by discussing interesting algebraic properties connecting
the operators with one another, and by providing a closed
form for the rank of an element.

In the rest of the paper we use the Hoare--Ramshaw
notation~\cite{GKPCM} for intervals in a partially ordered set: $[a\..b]=\{\?
x\mid a\leq x \leq b\?\}$ denotes a closed interval from $a$ to $b$,
$(\leftarrow\..b]=\{\?
x\mid x \leq b\?\}$ denotes an interval containing all elements less than or equal to $b$, and $[a\..\rightarrow)=\{\?
x\mid a\leq x\?\}$ denotes an interval containing all
elements greater than or equal to $a$.  
In case we need to treat indifferently intervals of the form $[a\..b]$ or
$(\leftarrow\..b]$, we will use the shortcut notation $[-\..b]$, and analogously
for $[a\..-]$. We also
write $\singint x$ as a shortcut for $[x\..x]=\{\?x\?\}$.

\section{A motivating example}

\begin{figure}
\begin{center}
\includegraphics[scale=.7]{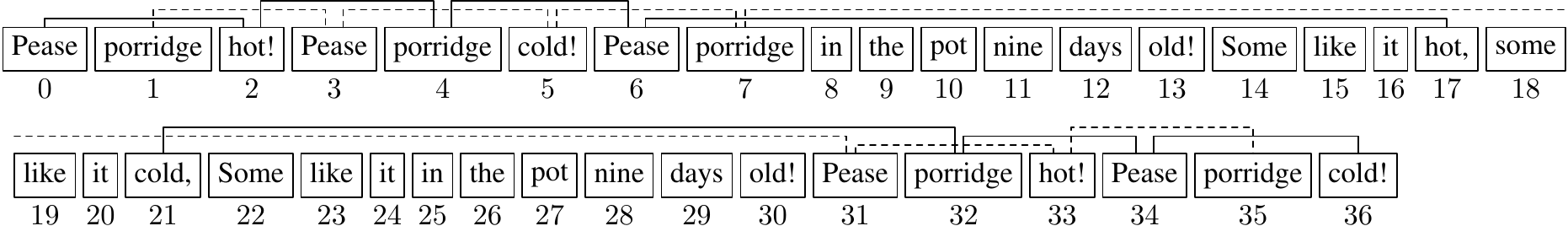}
\end{center}
\caption{\label{fig:example}A document; the intervals corresponding to the
semantics of the query ``\emph{porridge}
AND
\emph{pease} AND (\emph{hot} OR \emph{cold})'' are shown. For easier
reading, every other interval is dashed.}
\end{figure}

In this section we provide an introduction to minimal-interval semantics
by examples. Consider the text document represented in Figure~\ref{fig:example}.
Queries associated with a single keyword have a natural
semantics in terms of antichains of intervals---the list of positions where
the keyword occurs as singleton intervals.
For example, since the word ``hot'' appears only three times (in positions 2, 17 and 33), the
semantics of ``\emph{hot}'' will be
\[
\{\?[2\..2],[17\..17],[33\..33]\?\}.
\]
If we start combining terms disjunctively, we get simply the union of their
positions.
For instance, ``\emph{hot} OR \emph{cold}'' gives
\[
\{\?[2\..2],[5\..5],[17\..17],[21\..21],[33\..33],[36\..36]\?\}.
\]
If we consider the conjunction of two terms, we will start getting non-singleton
intervals: the semantics of ``\emph{pease} AND \emph{porridge}'' is computed by
picking all possible pairs of positions of \emph{pease} and \emph{porridge}
and keeping the minimal intervals among those spanned by such pairs:
\[
\{\?[0\..1],[1\..3],[3\..4],[4\..6],[6\..7],[7\..31],[31\..32],[32\..34],[34\..35]\?\}.
\]
The more complex query ``(\emph{pease} AND \emph{porridge}) OR \emph{hot}'' is
interesting because we have to take the intervals just computed, put them
together with the positions of \emph{hot}, and remove the non-minimal intervals:
\[
\{\?[0\..1],[2\..2],[3\..4],[4\..6],[6\..7],[17\..17],[31\..32],[33\..33],[34\..35]\?\}.
\]
One can see, for example, that the presence of \emph{hot} in position $2$ has
eliminated the interval $[1\..3]$.

Let's try something more complex: ``\emph{pease} AND \emph{porridge} AND (\emph{hot} OR \emph{cold})''.
We have again to pick one interval from each
of the three sets associated to ``\emph{pease}'', ``\emph{porridge}'' and ``\emph{hot} OR \emph{cold}'',
and keep the minimal intervals among those spanned by such triples (see
Figure~\ref{fig:example}):
\begin{multline*}
\{\?[0\..2],[1\..3],[2\..4],[3\..5],[4\..6],[5\..7],[6\..17],[7\..31],\\
[21\..32],[31\..33],[32\..34],[33\..35],[34\..36]\?\}.
\end{multline*}
From this rich semantic information, a number of different outputs can be
computed. A simple snippet extraction algorithm would
compute greedily the first $k$ smallest nonoverlapping intervals of the antichain, which would yield, for
$k=3$, the intervals $[0\..2]$, $[3\..5]$, $[31\..33]$, that is, ``\emph{Pease}
\emph{porridge} \emph{hot}!'', ``\emph{Pease} \emph{porridge} \emph{cold}!'',
and, again,  ``\emph{Pease} \emph{porridge} \emph{hot}!''.

A ranking scheme such as that proposed by Clarke and Cormack~\cite{ClCSSRR}
would use the number and the length of these intervals to assign a score to the document with
respect to the query. In a simplified setting, we can
assume that each interval yields a score that is the inverse of its length. The
resulting score for the query above would be
\begin{multline*}
\frac1{|[0\..2]|}+\frac1{|[1\..3]|}+\cdots+\frac1{|[6\..17]|}+\frac1{|[7\..31]|}+\cdots
\frac1{|[34\..36]|} \\
=\frac13+\frac13+\cdots+\frac1{12}+\frac1{25}+\cdots\frac13=\frac{177}{50}=3.54.
\end{multline*}
Clearly, documents with a large number of intervals are more relevant, and
short intervals increase the document score more than long intervals, as short
intervals are more informative.
The score associated to ``\emph{hot}'' would be just $3$ (i.e., the number of
occurrences). One can also incorporate positional information, making, for
example, intervals appearing earlier in the document more
important~\cite{BoVTREC2005}.

\section{Preliminaries of lattice theory}

In this section we are going to recall and briefly summarize some results from
lattice and order theory that we will need in the rest of the paper. The reader
can find more details, for example, in~\cite{ErnABCOT}. Note that in the previous sections we
used the word ``minimal'' always meaning ``by inclusion'', but in the following sections
its meaning will depend on the underlying order.

Let $P=(P,\leq)$ be a partially ordered set (poset). 
If $P$ has a minimum element, or bottom, $0$ (a maximum element, or top, $1$,
respectively) an \emph{atom}
(\emph{coatom}, respectively) of $P$ is an element $x \in P$ such that $0<
y\leq x$ ($x\leq y< 1$, respectively) implies $y=x$.
We say that $P$ is \emph{atomic} (\emph{coatomic}, respectively) if, for every
$x\neq 0$ ($x \neq 1$, respectively), there exists an atom (coatom,
respectively) $a$ such that $a \leq x$ ($x \leq a$, respectively); it is
\emph{strongly (co)atomic} if, for every two $x,y \in P$ such that $x < y$, the
poset induced by $[x\..y]$ is (co)atomic. 
We say that $P$ is \emph{atomistic}
(\emph{coatomistic}, respectively) if every $x\neq1$ ($x \neq0$,
respectively) is the $\vee$ ($\wedge$, respectively) of a set of atoms (coatoms, respectively). 
Note that (co)atomistic implies (co)atomic.
For a given poset $P=(P,\leq)$, we
let $P^\op$ denote the dual of $P$, that is, $(P,\leq)^\op=(P,\geq)$.

A \emph{$\vee$-semilattice}
(\emph{$\wedge$-semilattice}) is a poset that has all binary joins
(meets). A \emph{lattice} is a $\vee$-semilattice that is also a
$\wedge$-semilattice; it is 
\emph{bounded} if it has a top and a bottom.

We say that a lattice $L=(L,\leq)$ is a \emph{Heyting algebra}~\cite{JohSS} if
it is bounded and, for every $x$, the map $x \wedge \cdot$ has a right adjoint. In other words, for
every $x,z \in L$ there exists a greatest $y$ such that $x \wedge y \leq z$. Such a $y$ is denoted by $x \to z$ and called the
\emph{pseudo-complement of $x$ relative to $z$}. 

Dually, it is a \emph{Brouwerian algebra}~\cite{McTCECA}\footnote{There is some confusion in the literature about the usage of the 
adjective ``Brouwerian'': some authors call ``Brouwerian lattice'' 
what other authors call a ``Heyting algebra''.} if it is bounded and,
for every $x$, the map $x \vee \cdot$ has a left adjoint. In other words, for every
$x,z \in L$ there exists a smallest $y$ such that $z \leq x \vee y$. Such a $y$ is denoted by $x - z$ and called the \emph{pseudo-difference
between $x$ and $y$}.

For the
remainder of this section, $L=(L,\leq)$ will be a \emph{complete} lattice, that
is, a poset with $\vee$ and $\wedge$ for arbitrary sets. 
We say that $L$ is a \emph{completely distributive lattice} if, for every
collection $\L\subseteq 2^L$ of subsets of $L$, we have
\[
	\bigwedge\left(\bigvee \L\right)= \bigvee\left(\bigwedge
	\L^\sharp\right)
\]
where $\L^\sharp$ is the family of all subsets of $L$ that have a
nonempty intersection with all sets of $\L$. Every completely
distributive lattice is at the same time a Heyting and a Brouwerian algebra.

\smallskip
An element $x \in L$ is called:
\begin{itemize}
  \item \emph{\ORirr} (\ANDirr,
  respectively) if and only if $x=a\vee b$
  ($x=a\wedge b$, respectively) implies $x = a$ or $x = b$, for all $a,b\in L$;
  \item \emph{\ORprime} (\ANDprime, respectively) if and
  only if $x\leq a \vee b$ ($x\geq a \wedge b$,
  respectively) implies $x \leq a$ or $x \leq b$ ($x \geq a$ or $x \geq b$, respectively), for all $a,b\in L$;
  \item \emph{completely \ORirr} (\ANDirr,
  respectively) if and only if, for every $X \subseteq L$, $x=\bigvee X$
  ($x=\bigwedge X$, respectively) implies $x\in X$;
  \item \emph{completely \ORprime} (\ANDprime, respectively) if
  and only if, for every $X \subseteq L$, $x\leq\bigvee X$ ($x\geq\bigwedge X$,
  respectively) implies $x \leq a$ ($x\geq a$, respectively) for some $a \in X$.
\end{itemize}

In general, (complete) primality implies (complete) irreducibility. In
(completely) distributive lattices, the reverse holds, too.

We say that $L$ is \emph{superalgebraic} if every
element is the join of a set of completely \ORprime elements.
This notion is self-dual, that is, in a superalgebraic lattice every element is
also a meet of a set of completely \ANDprime elements. If a lattice is
completely distributive, of course, one can replace primality with
irreducibility in the definition above.

\section{The Alexandrov completion}
\label{sec:alex}

The \emph{Alexandrov completion} of a poset
$P=(P,\sqleq)$ is defined as $\L(P)=(\L(P),\subseteq)$
where $\L(P)$ is the family of lower sets\footnote{A set $X\subseteq P$ is a \emph{lower set} if
$x\sqleq y \in X$ implies $x \in X$. \emph{Upper sets} are defined dually.}
of $P$.

The nature of $\L(P)$ is well known.
Ern\'e~\cite{ErnABCOT} studied the interplay between Alexandrov completions of
posets, Alexandrov spaces (topological spaces where arbitrary unions of closed
sets is still closed) satisfying the $T_0$ separation axiom (for every pair of
points there is an open set containing exactly one of them), and completely
distributive lattices. In general, Alexandrov completions are 
superalgebraic~\cite[Proposition 2.2.B]{ErnABCOT}. However, we can prove
more if $P$ satisfies the \emph{ascending chain condition} (ACC), that is, if it
does not contain infinite chains of the form $x_0 \sqless x_1 \sqless x_2
\sqless \cdots$:
 \begin{theorem}[\protect{\cite[Lemma 5.12]{ErnABCOT}}]
	If $P$ satisfies the ACC then
	$\L(P)$ is a strongly coatomic, superalgebraic, 
	completely distributive lattice.\footnote{In fact, much more is true: every
	strongly coatomic, superalgebraic, completely distributive lattice is
	isomorphic to the Alexandrov completion of a poset satisfying the ACC. Moreover,
such lattices are exactly the lattices of closed sets of \emph{sober}
$T_0$ Alexandrov spaces~\cite{ErnABCOT}. By a natural choice of morphisms, this
correspondence can be made into a categorical equivalence.}
\end{theorem} 

The smallest lower set containing $X\subseteq P$ will be denoted by $\downset
X$ (analogously for the smallest upper set $\upset X$). Note that the map from
$P$ to $\L(P)$ sending an element $x$ to $\downset\sing x$, the \emph{principal ideal associated with $x$}, 
is an \emph{order embedding} (an injective function preserving and reflecting the
order relations) and preserves all meets existing in $P$.

\paragraph{(Completely) irreducible elements.}
Since Alexandrov completions are superalgebraic and completely distributive,
every element is the $\vee$ of a set of completely \ORirr
elements, and the $\wedge$ of a set of completely
\ANDirr elements.

The completely \ORirr elements in the Alexandrov completion
$\L(P)$ are in one-to-one correspondence with the elements of $P$, via the map
sending $x \in P$ to $\downset\sing x$. Since $\L(P)$ is strongly
coatomic, there is no difference between \ORirr and completely \ORirr
elements~\cite{CrDATL}. Moreover, each element has an
\emph{irredundant}\footnote{A \ORrepr $X$ of $x$ is \emph{irredundant} if no proper subset of $X$ can be used to represent $x$, that
is, $x =\bigvee X\neq \bigvee Y$ for every $Y\subset X$. Analogously, one
can define irredundant $\wedge$-representations.} \ORrepr by (completely) \ORirr
elements~\cite[Proposition~6.4]{CrDATL}. 

Also the completely $\wedge$-irreducible elements in 
$\L(P)$ are in one-to-one correspondence with the elements of $P$,
this time via the map sending $x \in P$ to $P\setminus \upset\sing x$.
If $P$ is coatomistic, the 
set $P\setminus \upset\sing x$ can be rewritten as the union of the
principal ideals $\downset\sing c$ for all the coatoms $c$ not above $x$. 
We remark,
however, that in this case there is no automatic equivalence between
$\wedge$-irreducible and completely $\wedge$-irreducible elements, as ${\mathscr
L}(P)$ is not strongly atomic in general. For the same reason,
some elements might not have an irredundant \ANDrepr by
completely \ANDirr elements.

\section{The antichain completion}
An \emph{antichain} of a poset $P=(P,\sqleq)$
is a subset $A \subseteq P$ whose elements are pairwise
$\sqleq$-incomparable.
The set of antichains of $P$ is denoted by $\A(P)$. It is possible to define a partial order on the antichains of
$P$ by letting, for all $A,B \in \A(P)$,
\[
  A \leq B \qquad\text{iff}\qquad \downset A\subseteq \downset B.
\]
We can equivalently describe (perhaps in a more direct way) the order of
$\A(P)$ as follows: given antichains $A$ and $B$,
\begin{equation}
\label{eq:alexorder}
A\leq B \qquad\text{iff}\qquad \forall x \in A \,\,\exists y \in B \quad x
\sqleq y.
\end{equation}

We call the partial order $\A(P)=(\A(P),\leq)$ the
\emph{antichain completion of $P$}. The map $\downset(-)$ defines an
 order embedding from $\A(P)$ to $\L(P)$, through which the
 embedding of $P$ into $\L(P)$ factors as
 $x\mapsto\sing x$:
\[\xymatrix{
(P,\sqleq) \ar[rrd]_{x\mapsto\downset\sing x}\ar[rr]^{x\mapsto\sing x} &&
(\A(P),\leq)\ar[d]^{A\mapsto\downset A}\\
&&(\L(P),\subseteq)}\]
 Note that, as any order embedding, $\downset(-)$ reflects all joins and meets: many proofs about
$\A(P)$ can be carried out by exploiting judiciously this property. In the main
object of study of this paper, the embedding $\downset(-)$ will be an
isomorphism (see Theorem~\ref{thm:alexandroff} below), but this is not true in
general.

In fact, $\A(P)$ may not even be a
lattice: if we consider, for example, a poset $P$ with elements
$a_0\sqless a_1\sqless a_2\sqless \cdots\sqless a_k\sqless \cdots$ and two additional, incomparable elements $b$
and $c$ that are greater than all of the $a_i$'s, it is easy to see that the meet $\sing b\sqcap\sing c$ does not
exist in $\A(P)$. However, $\A(P)$ is always a $\vee$-semilattice endowed
with unique $\lor$-representations by \ORirr elements:\footnote{The reader
should note that we defined \ORirr elements for \emph{lattices}, but in fact the definition is
sensible in any $\vee$-semilattice.}
\begin{theorem}
\label{th:antisup}
Let $P=(P,\sqleq)$ be a poset. The antichain completion $\A(P)$ is
a $\vee$-semilattice, where the join of $A, B\in \A(P)$ is given by the
maximal elements of $A\cup B$. Moreover, the \ORirr elements of
$\A(P)$ are the singleton antichains, and given $A\in\A(P)$
\[
A=\bigvee_{x\in A}\sing x
\]
is the unique irredundant $\lor$-representation of $A$ by \ORirr elements. 
\end{theorem}
\begin{proof}
The first statement is proved by noting that the maximal elements of $\downset
A\cup \downset B$ are exactly the maximal elements of $A\cup B$; since the lower
set of the maximal elements of $A\cup B$ is $\downset A\cup\downset B$, we
obtain the result by the fact that order embeddings reflect joins.
The second statement is a trivial consequence. For the third statement, note
that $A$ is equal to the maximal elements of $\bigcup_{x\in A}\downset\sing x$, and that 
irredundant $\lor$-representations by \ORirr elements (i.e., singleton
antichains) are in bijection with antichains of elements of $P$.
\end{proof}
The poset $a_0\sqless a_1\sqless a_2\sqless \cdots\sqless a_k\sqless \cdots$ shows that an
antichain completion is not necessarily $\vee$-complete: the join of the singleton antichains $\sing{a_i}$,
$i\in\N$, does not exist. On the other hand, the poset
$a_0\sqless a_1\sqless a_2\sqless \cdots\sqless a_k\sqless \cdots\sqless a$ shows that singleton antichains are not
necessarily \emph{completely} \ORirr, as $\sing a = \bigsqcup_{i\in\N}\sing{a_i}$.

\paragraph{The antichain completion as a lattice.} 
There is a simple necessary and sufficient condition that will turn the antichain completion into a lattice:
\begin{theorem}
\label{th:antiinf}
Let $P=(P,\sqleq)$ be a poset. Then, $\A(P)$ is a
lattice iff for every $A, B\in \A(P)$  there is a $C\in \A(P)$
such that
\begin{equation}
\label{eq:cap}
\downset A \cap \downset B=\downset C,
\end{equation}
and then $A\wedge B=C$; moreover, under this condition $\A(P)$ is
also distributive.
In particular, if $P$ is a $\sqcap$-semilattice and $\A(P)$ is a lattice the
meet of $A, B\in \A(P)$ is given by the maximal elements of $\{\?a\sqcap b\mid a \in A, b\in B\?\}$.
\end{theorem}
\begin{proof}
Condition~(\ref{eq:cap}) is sufficient for $\A(P)$ to be a lattice: indeed,
it is equivalent to the fact that the image of $\A(P)$ with respect to
$\downset(-)$ is a $\cap$-semilattice. Since we know already from Theorem~\ref{th:antisup} that it is a $\cup$-semilattice, we conclude that it is
a lattice, so the same is true of $\A(P)$. Since the image of $\A(P)$ is a set
lattice, it is distributive, and once again the same is true of
$\A(P)$.

Let us show that condition~(\ref{eq:cap}) is also necessary: suppose by
contradiction that $\A(P)$ is a lattice, but there are $A, B\in \A(P)$ such that
the lower set $X=\downset A \cap \downset B$ is not generated by an antichain.
In this case, $X$ contains an element $x$ that is not bounded by any maximal
element of $X$. 
Still there is a $C$ such that $C=A\wedge B$. Then necessarily
$C\subseteq\downset C\subset X$. Since both $\downset C$ and $X$ are lower sets,
this implies that there is an $x\in X$ that is strictly greater or
incomparable with any element of $C$: thus, $C<C \vee\sing x\leq A
\wedge B$, a contradiction.

For the second statement, recall that $\downset X=\downset Y$ implies that $X$
and $Y$ have the same maximal elements. Since it is immediate to show that
\[
\downset \{\?a\sqcap b\mid a \in A, b\in B\?\} = \downset A \cap \downset
B=\downset C
\]
and $C$ is an antichain, the thesis follows.
\end{proof}
\begin{figure}
\centering
\includegraphics[scale=1]{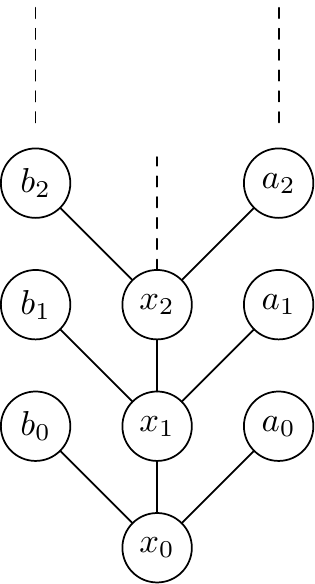}
\caption{\label{fig:nocap}An example in which condition~(\ref{eq:cap}) of
Theorem~\ref{th:antisup} is not satisfied, even assuming that existence of meets: 
$\downset\{\?a_i\mid i\in\N\?\}\cap\downset\{b_i\mid i\in\N\?\}=\{x_i\mid i\in\N\?\}$,
but the latter set is not a lower set generated by an
antichain.}
\end{figure}

Note that Condition~(\ref{eq:cap}) is weaker than the ACC (consider, for
example, the case in which $P$ is an ascending chain), and it is not tautological, 
as the example in Figure~\ref{fig:nocap} shows.
Moreover, even in the best
case of Theorem~\ref{th:antiinf} (i.e., $P$ a $\sqcap$-semilattice) $\A(P)$ may
not be complete. Consider the poset 
$c_{1}\sqless c_{2}\sqless c_{3}\sqless \cdots\sqless c_{k}\sqless \cdots\sqless c_{-k}\sqless \cdots\sqless c_{-3}\sqless c_{-2}\sqless c_{-1}$.
Clearly, $\bigsqcap_{i<0} \sing{c_i}$ and  $\bigsqcup_{i>0}
\sing{c_i}$ do not exist in $\A(P)$.

Under the ACC, however, one can prove much more:
\begin{theorem}[\protect{\cite[Corollary 10.5]{ErnEO}}]
	\label{thm:alexandroff}
	If $P$ satisfies the ACC, the order embedding $\downset(-):\A(P)\to
	\L(P)$ is an isomorphism.
\end{theorem}
Thus, in case $P$ satisfies the ACC the antichain completion of $P$
enjoys all strong properties of the Alexandrov completion of $P$. Alternatively, the
previous theorem shows that in the ACC case the antichain completion is
a handy representation of the Alexandrov completion.

\smallskip
\noindent{\bf Remark.} Theorems~\ref{th:antisup} and~\ref{th:antiinf} remain true for
the set of \emph{finite} antichains; moreover, in the case of finite antichains the hypothesis of
$P$ being a $\sqcap$-semilattice becomes sufficient for obtaining a lattice, because the 
set $\{\?a\sqcap b\mid a \in A, b\in B\?\}$ is finite and its maximal elements provide the finite antichain
$C$ that makes condition~(\ref{eq:cap}) true.\footnote{The hypothesis is not necessary, 
though. Consider the case of two incomparable elements $a$ and $b$ both smaller than
two incomparable elements $x$ and $y$: $x\sqcap y$ does not exist, but 
the set of finite antichains forms nonetheless a lattice.} The case of finite
antichains is interesting from a computational point of view, as discussed in~\cite{BoVEOLAMIS}.
Moreover, every finite-$\lor$-preserving map from the poset of finite antichains to a finitely
complete $\lor$-semilattice factors uniquely through the embedding
$\downset(-)$:
in other words, finite antichains are (isomorphic to) the free finitely
complete $\lor$-semilattice over their base poset~\cite{RibOSAOS}.

\section{Atoms and coatoms}

Let us first characterize the atoms of our completions:
\begin{proposition}
	Given a poset $P=(P,\sqleq)$, the atoms of the Alexandrov completion
	$\L(P)$ are exactly the lower sets of the form $\downset\sing x$, where
	$x$ is a minimal element of $P$; the atoms of the antichain completion $\A(P)$ are exactly the
	antichains of the form $\sing x$, where $x$ is a minimal element of $P$.
\end{proposition}
\begin{proof}
Consider an atom $A\subseteq P$ of $\L(P)$. If $A\neq \emptyset$
contains more than one element, pick $x\in A$ so that it is not the minimum of
$A$; note that $A \setminus\upset\sing x$ is still a lower set. Then,
$\emptyset\subset A\setminus\upset\sing x\subset A$, contradicting the fact that
$A$ is an atom. In the opposite direction, we just note that for a
 minimal $x\in P$ the set $\downset \{\?x\?\}$ contains just $x$, and thus covers $\emptyset$.
 The proof for $\A(P)$ is analogous.
\end{proof}

Since in a poset with a least element but no atoms every non-bottom element
is the start of an infinite descending chain, we have the following:
\begin{corollary}
	Consider a poset $P=(P,\sqleq)$. If $P$ has no minimal elements then
	each non-bottom element of $\A(P)$ or $\L(P)$ is the
	start of an infinite descending chain.
\end{corollary}

\smallskip
We now turn our attention to coatoms. Here we consider a scenario where $P$
satisfies the ACC (hence we state it directly for $\A(P)$) and is
coatomistic.

\begin{proposition}
	Let $P=(P,\sqleq)$ be a coatomistic poset satisfying the ACC, let $1_P$
	be its top element, and let $\operatorname{coat}(P)$ be the set of its coatoms.
	Then, the top element of $\A(P)$ is $\sing{1_P}$, 
	and the only coatom is $\operatorname{coat}(P)$. 
	Moreover, the elements $A$ of $\A(P)$ from which an 
	infinite ascending chain starts are exactly
	those for which $\operatorname{coat}(P)\setminus A$ is infinite, 
	that is, such that there are infinite	coatoms of $P$ not in $A$.
\end{proposition}
\begin{proof}
Note that $\operatorname{coat}(P)$ is by definition an antichain. Since every
element of $P$ except for $1_P$ is a meet of coatoms, $\operatorname{coat}(P)$
is greater than every other element of $\A(P)$ except for the top
$\sing{1_P}$, and thus the only coatom.

Now, let $N_A=\operatorname{coat}(P)\setminus A$ be the set of coatoms not in
$A$. If $N_A$ is infinite, say $N_A=\{\?c_0,c_1,c_2,\ldots\?\}$,
the sequence
\[
	A < A \vee \{\?c_0\?\} < A \vee \{\?c_0,c_1\?\} <
	\dots
\]
is an infinite ascending chain (the elements of the sequence are all distinct
because the $c_i$'s are always maximal).

Suppose by contradiction that $N_A$ is finite and that
\[
	A=A_0 < A_1 < A_2 < \dots
\]
is an infinite ascending chain starting from $A$. Note that the set of
non-coatoms in $A$ must be finite, as they must be the meet of elements from
$N_A$.

We can assume without loss of generality that at each step of the chain one of
the following two events happens:
\begin{itemize}
  \item a new element is added to $A$;
  \item a non-coatom in $A$ is substituted by a smaller element (and because
  of this, possibly some other elements are dropped from $A$).
\end{itemize}
Note that the first event can happen only a finite number of times, as such a
new element must be the meet of coatoms in $N_A$. Moreover, the second
event can happen only a finite number of times for a given element, or the ACC
would be violated. This contradicts the existence of the chain.
\end{proof}

\section{The \CCB lattice}

Let $O=(O,\sqleq)$ be a  totally ordered set. An \emph{interval} of $O$ is
any subset $I \subseteq O$ such that, for all $x,y,z \in O$, if $x\sqleq z\sqleq y$ and $x,y\in I$ then $z \in
I$. The set of all \emph{finite} intervals of $O$ is denoted by
$\I_O$.
If $\ell\sqleq r$, the interval $[\ell\..r]$ is nonempty, 
$\ell$ is its least element and
$r$ its greatest element; $\ell$ ($r$, respectively) will be called
the \emph{left} (\emph{right}, resp.) extreme of the interval.

The following proposition shows some properties of the set of
finite intervals ordered by \emph{reverse} inclusion $(\I_O,\supseteq)$ that are
relevant to its antichain completion:
\begin{proposition}
\label{prop:lfin}
Let $O=(O,\sqleq)$ be a totally ordered set. Then, the poset
$(\I_O,\supseteq)$ enjoys the following properties:
\begin{enumerate}
  \item it has a minimum iff $O$ is finite (and the minimum is $O$); if $O$ is
  infinite, $(\I_O,\supseteq)$ has no minimal element;
  \item\label{lfin:acc} it satisfies the ACC;
  \item \label{lfin:semsup} it is a $\vee$-semilattice, and the join of $I$ and
  $J$ is $I\cap J$;
  \item\label{lfin:semi} it is a $\wedge$-semilattice (and thus a lattice) iff
  $O$ is locally finite\footnote{A poset $(O,\sqleq)$ is locally finite iff for every $x,y\in O$ the interval $[x\..y]$ is
finite.},
  and the meet of $[\ell\.. r]$ and $[\ell'\.. r']$ is
\[[\ell\sqcap\ell'\..r\sqcup r'];\]
  \item\label{lfin:coat} it is coatomistic;
  \item\label{lfin:dimtwo} it is a subposet of the
  Cartesian product of two totally ordered sets, hence its dimension is at most
  $2$; it is exactly $2$ iff $|O|>1$.
\end{enumerate}
\end{proposition}
\begin{proof}
The first item is trivial. For (\ref{lfin:acc}) an infinite ascending
chain would be a sequence of finite intervals of the form $I_0\supset I_1\supset I_2\supset
\cdots$.

\noindent(\ref{lfin:semsup}) Trivial.

\noindent(\ref{lfin:semi}) If $O$ is not locally finite, there
exists an interval $[\ell\..r]$ of infinite cardinality, and
there exists no common $\supseteq$-lower bound to $\singint\ell$ and $\singint
r$ in $\I_O$. So $(\I_O,\supseteq)$ cannot be a
$\wedge$-semilattice.
On the other hand, if $O$ is locally finite the interval in the statement
is clearly the smallest finite interval containing both 
$[\ell\.. r]$ and $[\ell'\.. r']$.

\noindent(\ref{lfin:coat}) $\emptyset$ is the largest element of
$\I_O$. The elements just below it (the coatoms) are exactly the singleton
intervals, and each finite interval $[\ell\..r]$ is the meet of $[\ell]$ and
$[r]$.

\noindent(\ref{lfin:dimtwo}) Consider the poset $P_O=(O,\sqleq) \times
(O,\sqgeq)$; the injection $\iota: \I_O \to P_O$ sending $[\ell\..r]$ to
the pair $(\ell,r)$ respects the order. In fact, $\iota([\ell\..r]) \sqleq \iota([\ell'\..r'])$ if
and only if $\ell\sqleq\ell'$ and $r\sqgeq r'$, which happens precisely when
$[\ell\..r]\supseteq [\ell'\..r']$. Since $P_O$ is a subposet of the Cartesian
product of two total orders, its dimension is at most 2~\cite{OreTG}.
\end{proof}

Justified by the previous proposition, we assume from now on that
$O=(O,\sqleq)$ is a fixed locally finite, totally ordered set, so
$(\I_O,\supseteq)$ is a $\wedge$-semilattice, and we can use
Theorem~\ref{th:antiinf} to compute meets easily. Locally finite,
totally ordered sets correspond, up to isomorphisms, to the subsets of $({\mathbf Z},\leq)$.

The fact that $O$ is locally finite implies that it is strongly atomic and
coatomic. In particular, for every $x \in O$, either $x$ is the
greatest element of $O$ or there exists a single element $x' \in O$ such that $x\sqless x'$ and
$x\sqleq y\sqleq x'$ implies $y \in\sing{x,x'}$: the element $x'$
is called the \emph{successor} of $x$ and, when it exists, it is denoted by $x+1$.
Similarly, either $x$ is the least element of $O$ or there exists a single
element $x''\in O$ such that $x''\sqless x$ and $x''\sqleq y\sqleq x$ implies $y \in
\sing{x,x''}$: the element $x''$ is called the \emph{predecessor} of $x$ and,
when it exists, it is denoted by $x-1$.

Note that with a slight abuse of notation we will write
$[\ell+1\..\rightarrow)$ even when $\ell$ has no successor to mean the empty set. Analogously for  
$(\leftarrow\.. r-1]$ when $r$ has no predecessor. This convention makes it possible
to avoid the introduction of open or semi-open intervals.

We are now ready to state our main definition:
\begin{definition}
\label{def:ccbl}
Given a locally finite, totally ordered set $O$, the \emph{\CCB
lattice on $O$}, denoted by $\E_O$, is the antichain completion of $(\I_O,\supseteq)$.
\end{definition}

The above definition says that $\E_O$ is the set of antichains of
finite intervals with respect to inclusion, with partial order given by
\[
  A \leq B \qquad\text{iff}\qquad \forall I \in A \,\, \exists J \in
  B \quad J \subseteq I,
\]
which is just an explicit restatement of equation~(\ref{eq:alexorder}).
In other words, $A\leq B$ if for every interval $I$ in $A$
there is some better interval (witness) $J$ in
$B$, where ``better'' means that the new interval $J$ is contained in $I$. This
corresponds to the intuition that smaller intervals are
more \emph{precise}, and thus convey more information.

In fact, the lattice $\E_O$ extends the definition provided by Clarke,
Cormack and Burkowski~\cite{CCBASTSFI} in two ways: first, our base set can contain
also infinite (ascending, descending or bidirectional) chains; second,
we allow for  the empty interval. The first generalization makes it possible to
avoid specifying the document length, allowing for infinite virtual
documents, too. The second one gains us a top element that is useful in modeling
negation.\footnote{It is interesting to note that in our formulation
$\E_\emptyset$ is the Boolean lattice (containing only ``false'' and ``true''):
if there is no spatial information, minimal-interval semantics reduces to the Boolean case.}

We remark that in our setting each antichain can be totally ordered by left
(or, equivalently, right) extreme. This ordering (which we will call \emph{natural}) is locally finite, so it makes
sense to talk about the predecessor or successor of an interval in an
antichain. An interval that has no successor (predecessor, respectively) in
an antichain will be called the \emph{last} (\emph{first}, respectively)
interval of the antichain. With $\bigcup A$ we denote the union of
the intervals in $A$: in other words, the set of elements of $O$ covered by some
interval in $A$.

We usually write $\E_n$ instead of $\E_{\{\?0,1,\ldots,n-1\?\}}$. It is known
that $|\E_n|=C_{n+1}+1$, where $C_n$ is the $n$-th Catalan number~\cite[item 183]{StaCN}.\footnote{There is an off-by-one with respect to~\cite{StaCN}, due
to the fact the author considers nonempty intervals only.} There is a simple,
recursive CAT (Constant Amortized Time) algorithm (Algorithm~\ref{alg:enum})
that lists all antichains except for $\sing\emptyset$: the algorithm greedily
adds to a base antichain $A$ a new interval, assuming that the first interval in
natural order that can be added to $A$ has left margin $\ell$ and right margin
$r$. The correctness proof is an easy induction.

\begin{Algorithm}
\begin{tabbing}
\setcounter{prgline}{0}
\hspace{0.3cm} \= \hspace{0.3cm} \= \hspace{0.3cm} \= \hspace{0.3cm} \=
\hspace{0.3cm} \=\kill\\
\pl\>\PROCEDURE enumerate($A,\ell,r$) \BEGIN\\
\pl\>\>emit($A$);\\
\pl\>\>\FOR $i=\ell$ \TO $n-1$ \DO\\
\pl\>\>\>\FOR $j=\max\{\?i,r\?\}$ \TO $n-1$ \DO\\
\pl\>\>\>\>enumerate($A\cup\sing{[i\..j]},i+1,j+1$)\\
\pl\>\>\>\OD\\
\pl\>\>\OD\\
\pl\>\END.
\end{tabbing}
\caption{\label{alg:enum}The CAT algorithm enumerating all antichains
(except for $1$).
The base call is enumerate($\emptyset, 0, 0)$.}
\end{Algorithm}

The following theorem instantiates the results of Section~\ref{sec:alex} to $\E_O$,
using Proposition~\ref{prop:lfin}. The 
first two
items show that $\E_n$ is exactly the lattice defined 
in~\cite{CCBASTSFI}.
\begin{theorem}
$\E_O$ is a strongly coatomic, superalgebraic, completely distributive lattice.
Moreover, for any $A,B \in \E_O$, we have that:
 \begin{enumerate}
 \item\label{enu:sup} $A \vee B$ is the set of all
 $\subseteq$-minimal elements of $A \cup B$;
 \item\label{enu:inf} $A \wedge B$ is the set of all
 $\subseteq$-minimal elements of the set 
 \[
 	\{\?[\min\{\?\ell,\ell'\?\}\..\max\{\?r,r'\?\}] \mid [\ell\..r] \in A,
 	[\ell'\..r'] \in B\?\};
 \]
 \item\label{enu:bottom} the least element $0$ of $\E_O$ is $\emptyset$; 
 \item\label{enu:top} the greatest element $1$ of $\E_O$ is $\sing\emptyset$; no other element of
 $\E_O$ contains $\emptyset$;
 \item\label{enu:coatom} $\E_O$ has exactly one coatom, $\{\?\singint x \mid x \in O\?\}$, denoted by
 $1^-$;
 \item\label{enu:atom} if $O$ is finite, then $\E_O$ has exactly one atom, $\sing O$; no other
 element of $\E_O$ contains $O$;
 \item\label{enu:atominf} if $O$ is infinite, $0$ is not covered by any element
 (i.e., $\E_O$ has no atoms), and for each $A \neq 0$ there is an infinite
 descending chain starting at $A$;
 \item\label{enu:asc}  if $O$ is infinite, if $A\neq 1$ there is an infinite
 ascending chain starting at $A$ iff there are infinite singleton intervals
 not in $A$.
 \end{enumerate}
\end{theorem}

\section{Normal forms}

\paragraph{$\lor$-representations.}
Theorem~\ref{th:antisup} has already provided the unique \ORrepr of an antichain
$A$ by \ORirr elements, which we restate in our
case:
\begin{theorem}
\label{th:orrepr}
Let $A \in \E_O$. Then,
\[
A = \bigvee_{I\in A} \{\?I\?\}.
\]
and this is the only irredundant \ORrepr of $A$ by 
 \ORirr elements of $\E_O$.
\end{theorem}

\paragraph{$\land$-representations.}
A \ANDrepr in terms of completely
\ANDirr elements is also easy to write, as explained in Section~\ref{sec:alex},
but we immediately meet a computational issue:
if $O$ is infinite, a \emph{finite} antichain cannot be represented by a
finite set of completely \ANDirr elements, as the $\wedge$ of any such set
contains infinite singleton intervals. Moreover, when $O$ is infinite
$\E_O$ is not strongly atomic, and thus there are elements without an
irredundant \ANDrepr by completely \ANDirr elements~\cite[Proposition~6.3]{CrDATL}.

As we did for \ORirr elements, we thus turn to the study of 
\ANDirr elements and of the associated $\wedge$-representations, which, as we
will see, are also unique and irredundant; moreover, such
representations will be finite for finite antichains, and they will make it
possible to describe the relative pseudo-complement in closed form.

\begin{definition}
We denote with $\Inf_O$ the set of all (finite and
infinite) intervals of $O$. If $I\in\Inf_O$, we let
$\cgen I$ denote the set $\{\?\singint x \mid x \not\in I\?\}$, that is, the set
of all singletons that are not contained in $I$.
\end{definition}

The following proposition shows that the antichains of the form $\cgen I$ are
precisely the \ANDirr elements:
\begin{proposition}
An element $A\neq 1$ of $\E_O$ is \ANDirr iff $A=\cgen I$ for some
$I\in\Inf_O$.
\end{proposition}
\begin{proof}
Suppose that $A\neq 1$ and assume that there
is $[\ell\..r]\in A$ with $\ell\sqless r$. Since
$\{\?[\ell\..r]\?\}=\sing{[\ell]} \wedge \sing{[r]}$, we have,
letting $A'=A\setminus\{\?[\ell\..r]\?\}$, 
\[
A= A'\vee \{\?[\ell\..r]\?\} = 
A'\vee (\sing{[\ell]} \wedge \sing{[r]})=
(A'\vee\sing{[\ell]}) \wedge (A'\vee\sing{[r]}).
\]
Now we notice that $A'\vee\sing{[\ell]}$ cannot contain $[\ell\..r]$,
because it contains $[\ell]$, and the same is
true of $A'\vee\sing{[r]}$, so $A$ is not \ANDirr. 
Then necessarily $|I|=1$ for all $I\in A$. Now, suppose that
$(\bigcup A)^c$ is not an interval; then we have \[A= (A\cup \sing{\singint{x}})
\wedge (A \cup \sing{\singint y}),\] where $x$ and $y$ are taken so that they do not
belong to the same convex component of $(\bigcup A)^c$, and once again $A$ is not \ANDirr. 

Conversely, suppose by contradiction that $I\in \Inf_O$ and $A=\cgen I=A_1
\wedge A_2$. Then, for every $x \not\in I$ we must have $\singint x \in A_1 \cap
A_2$. Thus, the other intervals in $A_1$ and $A_2$ must be subintervals of $I$
(because $A_1$ and $A_2$ are antichains), and there must be at least one such interval
both in $A_1$ and in $A_2$. The meet of those two intervals 
is entirely contained in $I$, which implies that such interval, or a smaller one,
is in $A$, contradicting its definition.
\end{proof}

Figure~\ref{fig:e4} shows the lattice $\E_4$, with the
\ANDirr elements highlighted.
Note that in the infinite case the antichains $\cgen I$ with $I$ infinite are
exactly the \ANDirr elements that are \emph{not} completely \ANDirr.
\begin{figure}
	\centering
	\includegraphics[scale=.50]{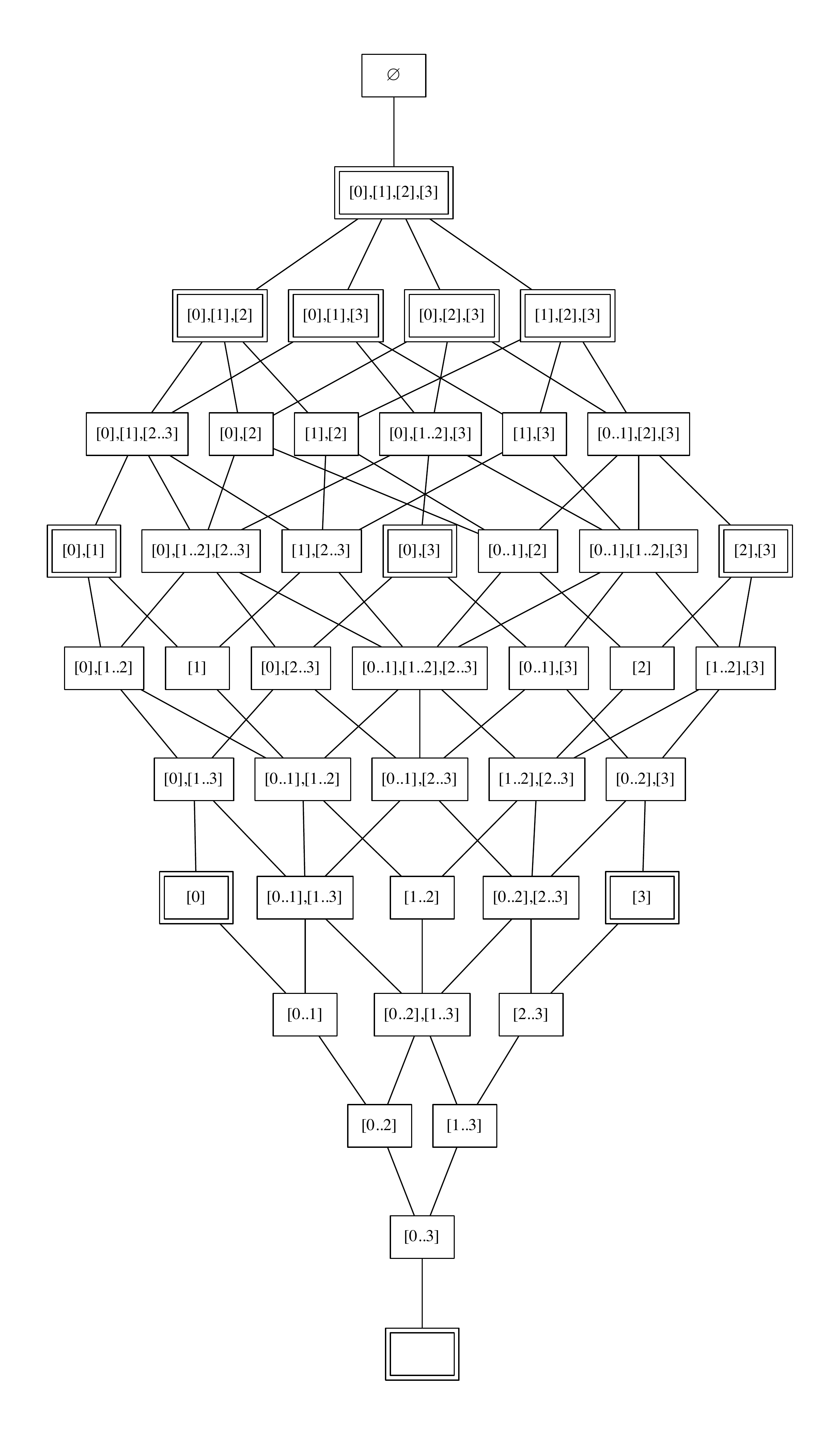}
	\caption{\label{fig:e4}The lattice $\E_4$. 
	The $\AND$-irreducible elements have a double border. Each
	horizontal layer corresponds to a level set (the set of elements of
	a given rank; see Section~\ref{sec:rank}).}
\end{figure}

We remark that $(\Inf_O,\supseteq)$ is a complete lattice
where the meet of a family of intervals is the smallest interval containing all
members of the family, and the join is just the intersection. Moreover, 
for any two intervals $I,J\in\Inf_O$ we have $\cgen I\leq\cgen J$ if
and only if $I\supseteq J$. In other words,
\begin{proposition}
\label{prop:isoandirr}
The map from $(\Inf_O,\supseteq)$ to the poset of \ANDirr
elements of $\E_O$ defined by the tilda operator is an isomorphism.
\end{proposition} 

The fact that \ANDirr elements correspond naturally to intervals of
$(\Inf_O,\supseteq)$ leads us to the following definition:
\begin{definition}
\label{def:crit}
An interval $I\in(\Inf_O,\supseteq)$ is \emph{critical} for an antichain
$A\in\E_O$ iff it is minimal among the
intervals of $(\Inf_O,\supseteq)$ that do not contain any interval of $A$.
The set of critical intervals for $A$ is denoted by
$\crit A$.
\end{definition}
Said otherwise, an interval is critical iff it is $\supseteq$-minimal in
$\Inf_O\setminus\downset A$, where the $\downset$ operator is computed in
$(\Inf_O,\supseteq)$ (as an antichain of $(\I_O,\supseteq)$ is also an antichain
of $(\Inf_O,\supseteq)$). 

The condition that a critical interval $I$ does not contain any interval of $A$
is trivially equivalent to $A \leq\cgen I$. By
Proposition~\ref{prop:isoandirr},
\begin{lemma}
Let $I \in \Inf_O$ and $A \in \E_O$. Then, $I$ is critical for $A$ iff
$A \leq \cgen I$ and $\cgen I$ is minimal with this property.
\end{lemma}
Thus, the critical intervals of $A$ are exactly those associated
with the minimal \ANDirr elements that dominate $A$. We now give an
explicit characterization, which also shows that $\crit A$ is finite if $A$ is finite:
\begin{theorem}
\label{th:crit}
An interval $I\in\Inf_O$ is critical for an antichain
$A$ iff one of the following happens:
\begin{itemize}
  \item $I=[\ell+1\..r'-1]\neq\emptyset$ and there are two intervals
  $[\ell\..r],[\ell'\..r']\in A$ such that the latter is the successor of the former in $A$;
  \item $I=\linf{r-1}\neq\emptyset$ and $[\ell\..r]$ is the first element of $A$
  for some $\ell$;
  \item $I=\rinf{\ell+1}\neq\emptyset$ and $[\ell\..r]$ is the last element of $A$
  for some $r$;
  \item $I=O$ and $A=0$;
  \item $I=\emptyset$ and $A=1^-$.
\end{itemize}
\end{theorem}

\begin{proof}
Assume first that $I\in\Inf_O$ is critical for $A$.
\begin{itemize}
\item If $I=\linf x$, then necessarily $x\sqless r$, where $r$ is the right
extreme of the first interval of $A$, otherwise the first interval would be included in
$I$. Thus, by $\supseteq$-minimality $x=r-1$; the case
$I=\rinf x$ is analogous.
\item Suppose that $I=[x\..y]$; since no interval of $A$ can be included in $I$,
every interval starts before $x$ or ends after $y$, and we can assume that there
is at least one interval starting before $x$ and at least
one interval ending after $y$ (otherwise we fall in the previous case). Let $[\ell\..r]$ be an
interval of $A$ starting before $x$, but with $\ell\sqless x$ as large as possible; since
it cannot be the last one, its successor $[\ell'\..r']$ will end after $y$
(i.e., $y\sqless r'$). But then by $\supseteq$-minimality $\ell=x-1$ and $y=r'-1$.
\end{itemize}
The other implication follows by a trivial case-by-case analysis.
\end{proof}

Our goal now is to show that the mapping
$A\mapsto \crit A$ is actually an isomorphism
\[
\A(\I_O,\supseteq)\cong\A(\Inf_O,\subseteq)^\op.
\]
More explicitly,
\[
\E_O=\A(\I_O,\supseteq)\cong \A\bigl((\Inf_O,\supseteq)^\op\bigr)^\op=
\A(\Inf_O,\subseteq)^\op.
\]
In particular, this isomorphism will imply that elements of $\E_O$ have
unique, irredundant $\land$-representations by antichains of \ANDirr elements.
Moreover, Theorem~\ref{th:antisup} and~\ref{th:antiinf} will provide the rules to perform
computations in $\A(\Inf_O,\subseteq)^\op$, and thus on
$\land$-representations.

We also remark that the antichain completion $\A(\Inf_O,\subseteq)^\op$ is \emph{not} an Alexandrov completion, because $(\Inf_O,\subseteq)$ does 
not satisfy the ACC in general. This consideration also explains why the
two dualizations do not cancel out. 
The reader 
should contrast this theorem with the fact that 
$\A(P)$ is isomorphic to the antichain completion of its 
(completely) \ORirr
elements.

The order in $\A(\Inf_O,\subseteq)^\op$ can be written
in elementary form by unwinding~(\ref{eq:alexorder}): given antichains $S,T$ 
of intervals in $\Inf_O$, we have the following chain of equivalences:
\[
S\leq T \text{ in $\A(\Inf_O,\subseteq)^\op$}\iff
T\leq S \text{ in $\A(\Inf_O,\subseteq)$}\iff
\forall I \in T \,\,\exists J \in S \quad I\subseteq J.
\]

As a first step towards proving the isomorphism, we now characterize the meets of incomparable \ANDirr elements in $\E_O$
in terms of the associated antichain in $(\Inf_O,\supseteq)$:
\begin{proposition}
\label{prop:infhat}
Consider an antichain $S$ of $(\Inf_O,\supseteq)$
and the meet (in $\E_O$) $M=\bigwedge\{\?\cgen I\mid I\in S\?\}$. Then, $M$ contains the
maximal intervals of $(\I_O,\supseteq)$ that are not contained in any interval
of $S$.
\end{proposition}
\begin{proof}
Note that $\downset M=\bigcap_{I\in S}\downset\cgen I$. Since $\downset\cgen I$
contains exactly all intervals that are not contained in $I$, $\bigcap_{I\in
S}\downset\cgen I$ contains exactly all intervals that are not contained in any
interval of $S$, and since $(\I_O,\supseteq)$ satisfies the ACC $M$ contains
exactly the maximal elements among such intervals.
\end{proof}
Said otherwise, an interval belongs to $M$ iff it is
$\supseteq$-maximal in $\I_O\setminus\upset S$. An immediate corollary
shows a special property of singleton intervals:
\begin{corollary}
\label{cor:sing}
Consider a nonempty antichain $S$ of $(\Inf_O,\supseteq)$
and the meet $M=\bigwedge\{\?\cgen I\mid I\in S\?\}$. Then, $[x]\in
M$ iff $x\not\in\bigcup S$.
\end{corollary}
The next theorem describes
explicitly the meets of incomparable \ANDirr elements, using the following notation:
\[
\ep{\ell,r}=\cgen{[\ell+1\..\rightarrow)}\wedge\cgen{(\leftarrow\..r-1]}=\begin{cases}
	\{\?[\ell\.. r]\?\} & \text{if $\ell\sqless r$,}\\
	\{\?\singint x \mid r \sqleq x \sqleq \ell\?\} & \text{otherwise.}
\end{cases}
\]
\begin{theorem}
\label{th:infhat}
Consider an antichain $S$ of $(\Inf_O,\supseteq)$
and the meet $M=\bigwedge\{\?\cgen I\mid I\in S\?\}$. Then, $M$ contains exactly
the following pairwise disjoint sets of intervals:
\begin{itemize}
  \item $\ep{\ell'-1,r+1}$, if $S$ contains two consecutive intervals
  $[-\..r]$, $[\ell'\..-]$.
  \item $\cgen{[\ell\..\rightarrow)}$ if $S$ has a first interval, and it is of
  the form $[\ell\..-]$.
  \item $\cgen{(\leftarrow\..r]}$ if $S$ has a last interval, and it is of the
  form $[-\..r]$.
  \item $\cgen\emptyset$ if $\emptyset\in S$.
  \item $\sing\emptyset$ if $S=\emptyset$.
\end{itemize}
\end{theorem}
\begin{proof}
 We will denote with
$M^*$ the set of intervals specified by the statement of the theorem, and with
$M$ the meet $\bigwedge\{\?\cgen I\mid I\in S\?\}$.
Proposition~\ref{prop:infhat} and a simple case-by-case analysis shows that $M^*
\subseteq M$. Moreover, Corollary~\ref{cor:sing} shows that the singleton
intervals of $M$ and $M^*$ are the same.

We are left to prove that no other non-singleton interval can belong to $M$
(hence $M^*=M$).
Assume by contradiction that there is an interval $I=[a\..b]\in M\setminus M^*$,
$a\sqless b$, which must be incomparable with all intervals in $M^*$
(because $M$ is an antichain, and $M^*\subseteq M$). Moreover, it cannot
be contained in any interval of $S$.

Since $I$ is not a singleton, by Corollary~\ref{cor:sing} it must be contained in
$\bigcup S$ (otherwise there would be an $x\in I$ with $\singint x\in M$, which
is impossible).
Consider the last interval $J=[\ell\..r]$ of $S$ with $\ell\sqleq a$. $J$
must necessarily overlap with $I$, but then $r\sqless b$, as $I$ cannot be
contained in any interval in $S$. Since $I\subseteq \bigcup S$, the interval $J$
must have a successor in $S$, say $[\ell'\..r']$, and of course $a
\sqless\ell'$, so $I$ contains $[\ell'-1\..r+1]\in M^*$, a contradiction.

A straightforward check shows that all the unions appearing in the description
of $M$ in the statement of the theorem are disjoint.
\end{proof}

We now fulfill our promise:
\begin{theorem}
\label{th:iso}
Given $A\in \E_O$ and $S\in \mathscr
A(\Inf_O,\subseteq)^\op$, the maps
\begin{align*}
A &\stackrel f\mapsto \crit A\\
S &\stackrel g\mapsto \bigwedge_{I\in S} \cgen I\\
\end{align*}
define an isomorphism between $\E_O$ and
$\A(\Inf_O,\subseteq)^\op$.
\end{theorem}
\begin{proof}
We start by proving monotonicity. Note that if
$S,T\in\A( \Inf_O,\subseteq)^\op$, if we have $S\leq T$ then for each $I\in
T$ there is a $J\in S$ such that $I\subseteq J$, which means $\cgen J\leq \cgen
I$. Thus, trivially $S\leq T$ implies $g(S)\leq g(T)$.

Suppose now we have $A,B\in \E_O$ with $A\leq B$, and let $I\in f(B)$.
Assume by contradiction that there is a $J\in A$ such that $J\subseteq I$. Since
$A\leq B$ there must be a $K\in B$ such that $K\subseteq J\subseteq I$,
contradicting the fact that, by definition, $I$ does not contain any interval of
$B$. Since $I$ does not contain any interval of $A$, it must be contained by definition
in a critical interval of $A$, that is, an interval in $f(A)$. We conclude that
$f(A)\leq f(B)$.

We now prove that $g\comp f=\mathbf 1$. By Definition~\ref{def:crit} and
Proposition~\ref{prop:infhat}, we have to show that
\[
g(f(A))=\max_\supseteq(\I_O\setminus \upset \min_\supseteq(\Inf_O\setminus
\downset A)) = A,
\]
where $\max_\supseteq$ ($\min_\supseteq$) computes the set of maximal (minimal)
elements by reverse inclusion.\footnote{Recall that when we write $\downset A$,
we are computing the lower set in $(\Inf_O,\supseteq)$ using the trivial
injection $(\I_O,\supseteq)\to(\Inf_O,\supseteq)$, which extends to antichains.} 

We first show that $\upset
\min_\supseteq(\Inf_O\setminus\downset A) = \Inf_O\setminus\downset A$. 
Let us start observing that given any family $X\subseteq \Inf_O$ whose
elements all contain a common interval, the union $\bigcup_{J
\in X} J$ is also an interval.
As a consequence, for every $I\in\Inf_O\setminus\downset A$  we have that
$\bigcup \{\?J 
\in \Inf_O\setminus \downset A \mid I\subseteq J\?\}$ is still an interval, so it is a
$\supseteq$-minimal element of $\Inf_O \setminus \downset A$.
In other words, $\Inf_O\setminus\downset A$ has the property that each element
dominates a $\supseteq$-minimal element, so taking 
the upper set of $\min_\supseteq(\Inf_O\setminus\downset A)$ gives back 
$\Inf_O\setminus\downset A$. Since $\I_O$ satisfies the ACC
\[
\max_\supseteq(\I_O\setminus \upset \min_\supseteq(\Inf_O\setminus
\downset A)) = \max_\supseteq(\I_O\setminus (\Inf_O\setminus
\downset A)) =  \max_\supseteq(
\I_O\cap\downset A ) = A.
\] 

Finally, we show that $f\comp g=\mathbf 1$. This is equivalent to
\[
f(g(S))=\min_\supseteq(\Inf_O\setminus \downset \max_\supseteq(\I_O\setminus
\upset S)) = S.
\]
We first show that $\downset
\max_\supseteq(\I_O\setminus\upset S) = \Inf_O\setminus\upset S$. Note that $\I_O\setminus\upset
S$ is a lower set of finite intervals and $\I_O$ satisfies the ACC, so taking in
$\I_O$ the lower set of $\max_\supseteq(\I_O\setminus\upset
S)$ would give back $\I_O\setminus\upset
S$. Building the lower set in $\Inf_O$, instead, we include also the
\emph{infinite} intervals of $\Inf_O$ that contain intervals of
$\I_O\setminus\upset S$, and this gives exactly $\Inf_O\setminus\upset S$. Thus,
\[
\min_\supseteq(\Inf_O\setminus \downset \max_\supseteq(\I_O\setminus
\upset S)) = \min_\supseteq(\Inf_O\setminus (\Inf_O\setminus
\upset S)) =  \min_\supseteq(
\upset S) = S.\qed
\]
\end{proof}

Thanks to the isomorphism we have just presented, we have an analogous of
Theorem~\ref{th:orrepr} for the case of $\land$-representations:
\begin{corollary}
Let $A \in \E_O$. Then,
\[
A = \bigwedge_{I\in \crit A} \cgen I
\]
and this is the only irredundant \ANDrepr of $A$ by 
\ANDirr elements of $\E_O$.
\end{corollary}
\begin{proof}
By Theorem~\ref{th:iso}, $g(f(A))=A$, so $\sing{\cgen I \mid I \in \crit A}$ is
an irredundant \ANDrepr of $A$. Suppose that
\[
A = \bigwedge_{I\in S} \cgen I,
\]
and that the representation is irredundant. Then, $S$ must be an antichain, so
$A=g(S)$, but then $\crit A=f(A)=f(g(S))=S$.
\end{proof}

\section{Relative pseudo-complement and pseudo-difference}

Using the well-known characterization of the
relative pseudo-complement in terms of \ANDirr elements~\cite{SmiMEIL}, we
have that given $A,B \in \E_O$, if $B=\bigwedge_i C_i$ is the \ANDrepr of $B$,
then
\begin{equation}
\label{eq:psc}
 A\to B= \bigwedge \{\?C_i\mid A\not\leq C_i\?\}.
\end{equation}
The right-hand side, albeit explicit, does not suggest an easy way to compute
$A\to B$, but we will provide a closed form in the next section.  As an
aside, this characterization implies that the negation induced by the
relative pseudo-complement is trivial: for every $A\neq 0$ we have $\lnot A =
A\to 0 = 0$, whereas $\lnot 0=1$.

By duality, we have that
given $A,B \in \E_O$, if $A=\bigvee_i C_i$ is the \ORrepr of $A$,
then
\[
 A - B= \bigvee \{\?C_i\mid C_i\not\leq B\?\}.
\]
We can further simplify this expression by noting that if $B=\bigvee_j D_j$ is
the \ORrepr of $B$, then $C_i\not\leq B$ iff there is no
$D_j$ such that $C_i\leq D_j$:
\[
 A - B= \bigvee \{\?C_i\mid \nexists D_j \text{ such that } C_i\leq D_j\?\}.
\]
Finally, since we know that such representations are simply given by sets of
singleton antichains we can just write
\begin{equation}
\label{eq:diff}
 A - B= \{\? I\in A\mid \nexists J \in B \text{ such that } J \subseteq I\?\},
\end{equation}
which characterizes pseudo-difference in $\E_O$ in a completely elementary way.
Note that the Brouwerian complement of $A$, $1-A$, is always equal to $1$, except for $1-1=0$.

\smallskip
It is customary to define a few derived operators, such as the \emph{symmetric
pseudo-difference}:
\[
A\mathbin\Delta B = (A-B)\vee(B-A).
\]
It is well known that in any Brouwerian algebra
\[
(A \vee B) - (A\wedge B)= A \mathbin\Delta B. 
\]
In our lattice, it is easy to verify that moreover
\begin{equation}
(A \vee B) - (A \mathbin\Delta B )= A\cap B, 
\end{equation}
where $\cap$ is the standard set intersection. In other words, intersection of
antichains can be reconstructed using only lattice operations. In particular,
this means that the $\AND$-semilattice of antichains ordered by set inclusion
can be be reconstructed, too, since $A\subseteq B$ iff $A\cap B=A$.

Another straightforward but important property we will
use is that
\begin{equation}
\label{eq:int}
A -( B\vee C) = (A-B)\cap(A-C). 
\end{equation}

We remark that the observations above hold true in any Brouwerian algebra
that has unique $\vee$-representations by \ORirr elements.
The only difference, in the general case, is that $(A \vee B) - A \mathbin\Delta
B$ will be equal to the join of the \ORirr elements in the intersection of the
representations of $A$ and $B$.

\subsection{An elementary characterization of the relative pseudo-complement}

The simplicity of the elementary characterization of the
pseudo-difference~(\ref{eq:diff}) is due the fact that we are implicitly
representing the elements of $\E_O$ by \ORirr elements. Characterizing 
the relative pseudo-complement would be analogously easy if we were using a
representation based on \ANDirr elements.

It is possible, however, to obtain an elementary characterization of the
relative pseudo-complement $A\to B$ based on \ORirr elements (i.e., on an
antichain of $\I_O$). To do that, first convert $B$ to a representation by
\ANDirr elements using Theorem~\ref{th:crit}; then, select those
elements that are not greater than $A$, as in~(\ref{eq:psc}). Finally, 
convert the set of selected elements to a representation by \ORirr
elements using Theorem~\ref{th:infhat}. In this section we will make this
process explicit when $B$ is finite.

Note that for $A\in \E_O$,
$I\in\Inf_O$ we have $A\not\leq\cgen I$ iff $\sing I\leq A$. Then, 
if $A,B \in \E_O$ with $A\neq1$, $B\neq 0,1$ finite\footnote{It is
immediate to verify that $1\to A=A$, $A\to 1 = 0\to0 = 1$, and that $A\to 0=0$
for $A\neq 0$.}, say, $B=\{\?[\ell_i\..r_i]\mid 0\leq i<n\?\}$, $A\to B$ is
the meet of the following \ANDirr elements:
\begin{itemize}
  \item $\cgen{(\leftarrow\..r_0-1]}$, if $\{\?\linf{r_0-1}\?\}\leq
  A$;
  \item $\cgen{[\ell_{i-1}+1\..r_i-1]}$, for $0<i<n$, if $\sing{[\ell_{i-1}+1\..r_i-1]}\leq
  A$;
  \item $\cgen{[\ell_{n-1}+1\..\rightarrow)}$, if  $\{\?\rinf{\ell_{n-1}+1}\?\}\leq A$.
\end{itemize}
Given this description, by a tedious but straightforward application of
Theorem~\ref{th:infhat} we can provide a description in terms of \ORirr elements:
\begin{theorem}
\label{th:psc}
Let $A,B \in \E_O$ with $A\neq1$, $B\neq 0,1$ finite, and
$B=\{\?[\ell_i\..r_i]\mid 0\leq i<n\?\}$. Define $T$ as the set of all $i$ with $0<i<n$
such that 
$\{\?[\ell_{i-1}+1\..r_i-1]\?\}\leq A$, and, if $T\neq\emptyset$, $\himin=\min
T$, $\himax = \max T$, $T^+ = T \setminus \{\?\himin\?\}$.
For every $i \in T^+$, let $P(i)$ denote the predecessor of $i$ in $T$
(i.e., the largest element of $T$ smaller than $i$).
Then, if $T\neq\emptyset$,
\[
	A\to B = U^- \cup \bigcup_{i \in T^+} \?\ep{\ell_{i-1},r_{P(i)}} \cup U^+
\]
where
\[
	U^- = \begin{cases}
	\ep{\ell_{\himin-1},r_0} & \text{if $\{\?\linf{r_0-1}\?\}\leq A$,}\\
	\cgen{[\ell_{\himin-1}+1\..\rightarrow)} & \text{otherwise,}
	\end{cases}
\]
\[
	U^+ = \begin{cases}
	\ep{\ell_{n-1},r_{\himax}} & \text{if $\{\?\rinf{\ell_{n-1}+1}\?\}\leq A$,}\\
	\cgen{(\leftarrow\..r_{\himax}-1]}& \text{otherwise.}
	\end{cases}
\]
If $T=\emptyset$,
\[
	A\to B = \begin{cases}
	\ep{\ell_{n-1},r_0} & \text{if $\{\?\rinf{\ell_{n-1}+1}\?\},\{\?\linf{r_0-1}\?\}\leq A$,}\\
	\cgen{\rinf{\ell_{n-1}+1}} & \text{otherwise, if $\{\?\rinf{\ell_{n-1}+1}\?\}\leq A$,}\\
	\cgen{\linf{r_0-1}} & \text{otherwise, if $\{\?\linf{r_0-1}\?\}\leq A$,}\\
	1 & \text{otherwise.}
	\end{cases}
\]
\end{theorem}

In the case of an infinite antichain $B$, if $T\neq\emptyset$ has no minimum
(maximum), one needs to eliminate $U^-$ ($U^+$, respectively) from the result.

\section{Containment operators}

Along the lines of Brouwerian difference, it is possible to define four
binary operators\footnote{A warning: although the notation
is reminiscent of binary \emph{relations}, the reader should
keep in mind that we are defining binary \emph{operations}.} that have
specific applications in information retrieval:
\begin{align}
\label{eq:cont}
 A \not\trianglerighteq B&= \{\? I\in A\mid \nexists\, J \in B \text{ such that }
 J \subseteq I\?\} = A\setminus \downset B\\
 A \trianglerighteq B&= \{\? I\in A\mid \exists\, J \in B \text{ such that } J
 \subseteq I\?\} = A\cap\downset B\\
 A \not\trianglelefteq B&= \{\? I\in A\mid \nexists\, J \in B \text{ such
 that } J \supseteq I\?\} = A\setminus \upset B\\
\label{eq:lastcont}
 A \trianglelefteq B&= \{\? I\in A\mid \exists\, J \in B \text{ such that } J
 \supseteq I\?\} = A\cap\upset B
\end{align}
These operators are named ``not containing'', ``containing'', ``not contained
in'' and ``contained in'', respectively, in~\cite{CCBASTSFI}. They are an
essential tool in finding part of a text satisfying further positional constraints: for example, if $K$
denotes the antichain of intervals of text where a certain set of keywords appears, and $T$ denotes the antichain of
intervals specifying which parts of the text are titles, $K\trianglelefteq T$
contains the intervals where the keywords appear inside a title.

 Note that:
\begin{eqnarray*}
 A \not\trianglerighteq B&=&A-B\\
 A \trianglerighteq B&=&A-(A-B).
\end{eqnarray*}
The first equality is due to~(\ref{eq:diff}), whereas the second follows from observing 
that the pseudo-difference with a subset is just complementation.
Other operators definable by pseudo-difference are the strict versions of
$\not\trianglerighteq$ and $\trianglerighteq$:
\begin{align*}
\label{eq:cont}
 A - (B - A) &=  A \notstrcont B= \{\? I\in A\mid \nexists\, J \in B
 \text{ such that } J \subset I\?\}\\
 A - ( A - (B - A)  )&=  A \mathbin\rhd B= \{\? I\in A\mid \exists\, J \in B
 \text{ such that } J \subset I\?\}
\end{align*}
These operators are important because $A\notstrcont B$ contains exactly the
intervals in $A$ that do not disappear by minimization in $A\vee B$. In a formula,
\begin{equation}
\label{eq:gsupcup}
A\vee B = (A\notstrcont B) \vee (B\notstrcont A)= (A\notstrcont B) \cup (B\notstrcont A).
\end{equation}
The two operators $\trianglelefteq$ and $\not\trianglelefteq$  do not seem to admit an easy description 
in terms of lattice operators, albeit we can always, of course, resort to
unique $\lor$-representations by \ORirr elements: the
definitions~(\ref{eq:cont})-(\ref{eq:lastcont}) can be indeed rewritten in any lattice in which such a representation exists, reading
the $\in$ symbol as ``is an element of the \ORrepr'' and
replacing containment with $\leq$.
 
Since these operators are defined by taking a subset of elements (of the
\ORrepr) of the first operand that depends only on the second
operand, any chain of applications of these operands can be permuted without affecting the results. In other words,
as noted in~\cite{CCBASTSFI},
\[
(A\ominus B)\oplus C = (A\oplus C)\ominus B\quad \ominus,\oplus \in
\{\?\trianglerighteq,\trianglelefteq,\not\trianglerighteq,\not\trianglelefteq\?\}.
\]
We note that an immediate consequence of $\not\trianglerighteq$ being a
lower adjoint is that
\[
(A\vee B)\not\trianglerighteq C =(A\not\trianglerighteq C )\vee
(B\not\trianglerighteq C ).
\]
This property is a form of distributivity. Some pseudo-distributivity properties are listed
in the following:
\begin{theorem}
\label{th:qdist}
Let $A,B,C\in \E_O$. Then,
\begin{align*}
A \not\trianglerighteq (B\wedge C) = (A \not\trianglerighteq B) \vee (A
\not\trianglerighteq C)\\
A \not\trianglerighteq (B\vee C) = (A \not\trianglerighteq B) \cap (A
\not\trianglerighteq C)\\
A \trianglerighteq (B\wedge C) = (A \trianglerighteq B) \cap (A
\trianglerighteq C)
\end{align*}
\end{theorem}
\begin{proof}
The first equality is well known~\cite{McTCECA}; the second was observed in~(\ref{eq:int}).
For the third one,
\begin{align*}
A \trianglerighteq (B\wedge C) & = A - ( A - (B\wedge C) )\\
&= A - (( A - B) \vee (A - C) )\\
&= (A - ( A - B))\cap (A-(A - C) )\\
&= (A \trianglerighteq B)\cap (A\trianglerighteq C).\qed\\
\end{align*}
\end{proof}

Finally, other distributivity properties
involving the containment operators hold: 
\begin{theorem}
\begin{enumerate}
  \item\label{enu:distr1} $\trianglelefteq$ and $\not\trianglerighteq$ are left-distributive over $\vee$; that is, 
  for every $A,B,C\in \E_O$ and $\oplus \in \{\?\trianglelefteq,\not\trianglerighteq\?\}$ we have 
	\[
	(A\vee B)\oplus C =(A\oplus C )\vee (B\oplus C ).
	\]
  \item\label{enu:distr2} $\trianglerighteq$ is right-distributive over $\vee$;
  that is, for every $A,B,C\in \E_O$ we have
	\[
	A\trianglerighteq (B \vee C) = (A\trianglerighteq B )\vee (A\trianglerighteq
	C ).
	\]
  \item\label{enu:distr3} No other distributive property of any of the operators $\{\?\trianglelefteq,\not\trianglerighteq,\trianglerighteq,\not\trianglelefteq\?\}$
  over any of $\{\?\vee,\wedge\?\}$ holds.
\end{enumerate}
\end{theorem}
\begin{proof}
For (\ref{enu:distr1}), as we observed above, the case
$\oplus=\mathbin{\not\trianglerighteq}$ depends on $\not\trianglerighteq$ being
a lower adjoint. 

Let us prove the case $\oplus=\mathbin\trianglelefteq$. The
left-hand side is formed by the set of minimal\footnote{From this section,
``minimal'' will always mean minimal by inclusion, that is,
$\supseteq$-maximal.} intervals in $A\cup B$ that are contained in some interval of $C$. Such intervals are either in $A\trianglelefteq C$ or $B\trianglelefteq C$, and they are obviously minimal in $(A\trianglelefteq C)\cup(B\trianglelefteq C)\subseteq A \cup B$.
Thus,
\[
(A\vee B)\trianglelefteq C \subseteq (A\trianglelefteq C )\vee
(B\trianglelefteq C ).
\]
The right-hand side is made by minimal intervals in $(A\trianglelefteq C
)\cup (B\trianglelefteq C)$. Let $I$ be such an interval, and assume without loss of
generality that it belongs to $A$. We have necessarily that $I\in A\vee B$, for
otherwise there should be an interval $J\in B$ such that $J\subset I$. Such an
interval would be contained \textit{a fortiori} in some interval of $C$, and
this contradicts the minimality of $I$.

For (\ref{enu:distr2}),
\begin{align*}
A\trianglerighteq (B \vee C) &= A - ( A - (B\vee C ) ) \\
&= A - ( (A-B)\cap
(A-C))\\
 	&= A - ((A-B)\wedge(A-C))\\
 	&= (A - (A-B))\vee
(A-(A-C))\\
&=(A\trianglerighteq B )\vee (A\trianglerighteq C ),
\end{align*}
where we used Theorem~\ref{th:qdist}, and the
fact that intervals in $(A-B)\wedge(A-C)$ but not in $(A-B)\cap
(A-C)$ are spans of two distinct intervals of $A$, so they cannot be included in
an interval of $A$.

For (\ref{enu:distr3}), there is one counterexample for each instance in
Table~\ref{tab:counterexamplessup} and~\ref{tab:counterexamplesinf}.
\end{proof}
\begin{table}
\begin{center}
\renewcommand{\arraystretch}{1.7}
\begin{tabular}{r|c|c}
$\oplus$ &$A\oplus(B\vee C)$ & $(A\vee B)\oplus C$\\
\hline 
$\not\trianglerighteq$& $A=C=\sing{[a]}$, $B=\sing{[a\..b]}$ & --- \\
$\trianglerighteq$ & --- & $A=C=\sing{[a\.. b]}$, $B=\sing{[a]}$\\
$\not\trianglelefteq$ & $A=B=\sing{[a]}$, $C=\sing{[b]}$ & $A=\sing{[a\.. b]}$,
$B=C=\sing{[a]}$\\
$\trianglelefteq$ & $A=B=\sing{[a\.. b]}$, $C=\sing{[a]}$ & --- \\
\end{tabular}
\caption{\label{tab:counterexamplessup}Counterexamples to missing distributivity
laws for containment operators over $\vee$.}
\end{center}
\end{table}

\begin{table}
\begin{center}
\renewcommand{\arraystretch}{1.7}
\begin{tabular}{r|c|c}
$\oplus$& $A\oplus(B\wedge C)$ & $(A\wedge B)\oplus C$\\
\hline 
$\not\trianglerighteq$ & $A=\sing{[a],[b]}$, $B=\sing{[a]}$, $C=\sing{[b]}$ & $A=\sing{[a]}$, $B=\sing{[b]}$, $C=\sing{[c]}$, $a\sqless c \sqless b$\\
$\trianglerighteq$ &  $A=\sing{[a],[b]}$, $B=\sing{[a]}$, $C=\sing{[b]}$& $A=\sing{[a]}$, $B=\sing{[b]}$, $C=\sing{[c]}$, $a \sqless c \sqless b$\\
$\not\trianglelefteq$ & $A=\sing{[a\.. b]}$, $B=\sing{[a]}$, $C=\sing{[b]}$& $A=\sing{[a]}$, $B=\sing{[b]}$, $C=\sing{[a],[b]}$\\
$\trianglelefteq$ & $A=\sing{[a\.. b]}$, $B=\sing{[a]}$, $C=\sing{[b]}$ & $A=\sing{[a]}$, $B=\sing{[b]}$, $C=\sing{[a],[b]}$
\end{tabular}
\caption{\label{tab:counterexamplesinf}Counterexamples to distributivity
laws for containment operators over $\wedge$.}
\end{center}
\end{table}


\section{Ordered operators}

For completeness, we introduce two final operators that are fundamental in 
information retrieval. For $A,B\neq 1$ the \emph{ordered non-overlapping meet operator} 
\[
A\mathbin<B=\bigvee\bigl\{\?\{\?[\min I\..\max J]\?\} \mid I\in A, J\in B, \max I
\sqless \min J\?\bigr\}
\]
returns the minimal intervals spanned by an interval of $A$ followed (without overlaps) by an interval of $B$.
Such an operator is useful when looking, for example, for a sequence of terms in
a given order in a window of size $k$. It suffices to compute the antichain
generated by the operator, and check whether it contains some interval
of length at most $k$.

The \emph{block operator}
\[
A\mathbin\square B = \{\?[\min I\..\max J] \mid I\in A, J\in B, \max
I+1=\min J\?\}
\]
returns the intervals formed by two consecutive intervals from $A$ and $B$ (such intervals form an antichain).
The block operator is used to implement phrase search. We define $1$ to be the identity for both operators.

Both operators are associative, and moreover the ordered meet distributes
with respect to joins: more precisely,
\begin{theorem}
We have
\begin{align*}
(A\vee B) \mathbin< C &= (A \mathbin<C)\vee (B\mathbin< C)\\ 
A \mathbin< (B \vee C) &= (A \mathbin< B)\vee (A \mathbin<C) 
\end{align*}
\end{theorem} 
\begin{proof}
If any of the antichain is $1$, the result is trivial. We prove $(A\vee
B)<C=(A\mathbin<C)\vee(B\mathbin<C)$.
The other case is similar.

Consider an interval $I$ in $(A\vee B)<C$. It is formed by the span of an interval that we
can assume without loss of generality to belong to $A$, and of an interval in $C$. We know that
$I$ is minimal among all spans between elements of $A\vee B$ and $C$. 
If $I$ does not belong to $A\mathbin<C$, this means that there is
another smaller span, which contradicts minimality in $(A\vee B)<C$. If $I$
belongs to $A\mathbin<C$ but not to $(A\mathbin<C)\vee(B\mathbin<C)$, it means that there's a smaller span in $(B\mathbin<C)$, again contradicting
the minimality of $I$. So $(A\vee B)<
C\subseteq(A\mathbin<C)\vee(B\mathbin<C)$

Consider now an interval $J$ in $(A\mathbin<C)\vee(B\mathbin<C)$, and assume
without loss of generality that it comes from $A\mathbin<C$. If $J$ does not
belong to $(A\vee B)<C$, it means that there is a smaller, minimal interval in $(A\vee B)<C$. But such an interval
is \textit{a fortiori} minimal in $(A\mathbin<C)$ or $(B\mathbin<C)$, which
contradicts the minimality of $J$.
\end{proof}

The example $A = \{\?[0\..1]\?\}$, $B = \{\?[0]\?\}$, $C = \{\?[2]\?\}$
shows that the block operator does not satisfy the first distributivity law of
the theorem; the example  $A = \{\?[0]\?\}$, $B = \{\?[1\..2]\?\}$, $C =
\{\?[2]\?\}$ that it does not satisfy the second one.

\section{Ranking and height}
\label{sec:rank}

Every finite distributive lattice $L$ is \emph{ranked}: there is a rank function
$\frak r:L\to \N$ such that $\frak r(x)=\frak r(y)+1$ whenever $x$ covers $y$, and
$\frak r(0)=0$. For a finite distributive lattice, the rank of an element is
the number of \ORirr elements it dominates. This characterization makes the
explicit computation of the rank very simple, if $O$ is finite; from now on, we
let $O=\{\?0,1,\dots,n-1\?\}$.
\begin{theorem}
Let $[\ell\..r]\in \I_O$; then,
\[
\frak r(\{\?[\ell\..r]\?\}) = (1+\ell)(n-r).
\]
Moreover, let $k>0$ and $X=\{\? [\ell_0\..r_0],\dots,[\ell_k\..r_k] \?\}\in\E_O$ with
$\ell_0\sqless\ell_1\sqless\dots\sqless\ell_k$. Then,
\[
	\frak r(X)=(1+\ell_0)(n-r_0)+\sum_{i=1}^k (\ell_i-\ell_{i-1})(n-r_i).
\]
\end{theorem}
\begin{proof}
Since \ORirr elements are singleton antichains, to rank $\{\?[\ell\..r]\?\}$ we
must compute the number of intervals that contain $[\ell\..r]$, which are
exactly those of the form $[\ell'\..r']$ with $\ell'\sqleq\ell$ and
$r\sqleq r'$, resulting in the first formula above. For the second formula,
we recall that in every distributive lattice
\[
\frak r(A)+\frak r(B)=\frak r(A \vee B) + \frak r(A\wedge B);
\]
taking 
$A=\{\? [\ell_0\..r_0],\dots,[\ell_{k-1}\..r_{k-1}] \?\}$,
$B=\{\?[\ell_k\..r_k]\?\}$ and observing that $A\wedge
B=\{\?[\ell_{k-1}\..r_k]\?\}$, we have that
\[
\frak r(X)=\frak r(\{\?[\ell_0\..r_0],\dots,[\ell_{k-1}\..r_{k-1}] \?\}) +
\frak r([\ell_k\..r_k]) - \frak r([\ell_{k-1}\..r_k]).
\]
Applying the same idea $k$ times, we have
\begin{multline*}
\frak r(X)=\sum_{i=0}^k \frak r(\{\?[\ell_i\..r_i]\?\})- \sum_{i=1}^k
\frak r(\{\?[\ell_{i-1}\..r_i]\?\})\\
=\sum_{i=0}^k (1+\ell_i)(n-r_i)- \sum_{i=1}^k
\frak
(1+\ell_{i-1})(n-r_i),
\end{multline*}
whence the result.
\end{proof}

\begin{corollary}
\label{cor:height}
The rank of the top element of $\E_n$ is $1+n(n+1)/2$. Hence $\E_n$
has height $2+n(n+1)/2$.
\end{corollary}
\begin{proof}
The rank of $1^-=\{\?[0],[1],\ldots,[n-1]\?\}$ is
 \[
  \frak r(1^-)=n+\sum_{i=1}^{n-1} (n-i) =n(n+1)/2;
 \]
the statement follows immediately.
\end{proof}

\section{Conclusions and open problems}

We have presented a detailed analysis of the \CCB lattice on a locally finite,
totally ordered set. Besides analyzing its basic properties, we have provided
elementary characterizations for all operations, including relative pseudo-complement and
pseudo-difference, and a closed formula for the rank of an element.

It is easy to check that join and meet of two elements are computable in
linear time using simple greedy algorithms. More efficient algorithms for the 
computation of $n$-ary joins and meets, as
well as algorithms for the containment and ordered operators, are described
in~\cite{BoVEOLAMIS}. Finally, in the description of $A\to B$ in the statement of Theorem~\ref{th:psc} all
intervals appear exactly once, and moreover they are enumerated in their
natural order.
This implies that in the finite case it can be turned into
an $O(|A|+|B|+|A\to B|)$ algorithm
to compute the relative pseudo-complement: first, an easy $O(|A|+|B|)$ greedy
algorithm computes $T$. Then, depending on the case conditions (which can be all evaluated in $O(|A|)$) a simple 
loop emits the output (which cannot take more time than $O(|A\to B|)$). Note
that in this particular case we have to state explicitly that the algorithm is
linear \emph{both in the input and in the output size}, as the output might
have roughly the same size as the base set even for a constant-sized
input (e.g., $\sing{\singint{n-1}}\to\sing{\singint{n-2}}=\{\?\singint0,
\singint1,\ldots, \singint{n-2}\?\}$).\footnote{If we admit a symbolic
representations for infinite sets of the form $\cgen{\rinf\ell}$ and $\cgen{\linf r}$, the time bound
becomes $O(|A|+|B|)$ even when $O$ is infinite, as long as the input antichains
are finitely representable.} All these algorithms have been implemented in
LaMa4J (Lattice Manipulation for Java), a free Java library supporting computation in lattices.\footnote{\texttt{\small http://lama4j.di.unimi.it/}}

An interesting open question is that relative to the \emph{width} of the lattice
$\E_n$, that is, the cardinality of a maximum antichain. The sequence of
the first few widths of $\E_n$ for $n=0,1,\ldots$ is
\[1,
2,
3,
7,
17,
44,
118,
338,
1003,
3039,
9466,
30009,\ldots,
\]
which does not appear in the ``On-Line Encyclopedia of Integer
Sequences''~\cite{OEIS}.
At these sizes, the width coincide with the cardinality of the maximum \emph{level
sets} (a level set is the set of elements of given rank). The fact that this
happens at all sizes is known at the \emph{Sperner property}~\cite{EngST}. It
would be interesting to find a closed form for the sequence above, and prove or disprove
the Sperner property for $\E_n$.  
Finally, it would be interesting to relax the totality assumption on the base
set $O$, to include, for instance, tree-structured documents. 

\section{Acknowledgements}

We are very grateful to the anonymous reviewer and to the editor: their
suggestions improved enormously the quality of the presentation.

\bibliography{biblio}

\end{document}